\documentclass{article}

\usepackage{amsmath,amsthm,amsfonts,amssymb,enumerate,hyperref,cleveref,color,multirow}
\usepackage{fullpage}

\newtheorem{theorem}{Theorem}[section]
\newtheorem{lemma}[theorem]{Lemma}
\newtheorem{corollary}[theorem]{Corollary}

\theoremstyle{definition}

\title{Junta threshold for low degree Boolean functions on the slice}
\author{Yuval Filmus}

\begin{document}

\maketitle

\begin{abstract}
We show that a Boolean degree~$d$ function on the slice $\binom{[n]}{k}$ is a junta if $k \geq 2d$, and that this bound is sharp. We prove a similar result for $A$-valued degree~$d$ functions for arbitrary finite $A$, and for functions on an infinite analog of the slice.
\end{abstract}

\section{Introduction} \label{sec:introduction}

A classical result of Nisan and Szegedy~\cite{NisanSzegedy} states that a Boolean degree~$d$ function on the Boolean cube $\{0,1\}^n$ is an $O(d2^d)$-junta.  Let us briefly explain the various terms involved:
\begin{itemize}
\item A function $f$ on the Boolean cube is \emph{Boolean} if $f(x) \in \{0,1\}$ for all $x \in \{0,1\}^n$.
\item A function $f$ on the Boolean cube has \emph{degree (at most) $d$} if there is a polynomial $P$ of degree at most $d$ in $n$ variables such that $f(x_1,\ldots,x_n) = P(x_1,\ldots,x_n)$ for all $x_1,\ldots,x_n \in \{0,1\}$.
\item A function $f$ is an \emph{$m$-junta} if there are $m$ indices $1 \leq i_1, \ldots, i_m \leq n$ and a function $g\colon \{0,1\}^m \to \mathbb{R}$ such that $f(x_1,\ldots,x_n) = g(x_{i_1},\ldots,x_{i_m})$.
\end{itemize}
Chiarelli, Hatami and Saks~\cite{CHS} improved the bound to $O(2^d)$, and the hidden constant was further optimized by Wellens~\cite{Wellens}.

The \emph{slice} $\binom{[n]}{k}$, also known as the \emph{Johnson scheme} $J(n,k)$, consists of all vectors in $\{0,1\}^n$ of Hamming weight $k$. Is it the case that all Boolean degree~$d$ functions on the slice $\binom{[n]}{k}$ are $m(d)$-juntas, for some constant $m(d)$?
Two partial answers to this question appear in~\cite{FI2019a,FI2019b}. First, if $f$ is a Boolean degree~$1$ function on $\binom{[n]}{k}$ and $k,n-k \geq 2$ then $f$ is a $1$-junta~\cite{FI2019a}. Second, there exist constants $C(d) = O(2^d)$ such that if $f$ is a Boolean degree~$d$ function on $\binom{[n]}{k}$ and $k,n-k \geq C(d)$, then $f$ is an $O(2^d)$-junta~\cite{FI2019b}.

The reason that both of these results require both $k$ and $n-k$ to be large is that given a function $f$ on $\binom{[n]}{k}$, we can construct a dual function $\bar{f}$ on $\binom{[n]}{n-k}$ with similar properties by defining $\bar{f}(x_1,\ldots,x_n) = f(1-x_1,\ldots,1-x_n)$. For this reason, when we consider the slice $\binom{[n]}{k}$, we typically assume that $n \geq 2k$.

One of the open questions in~\cite{FI2019b} asks for the minimal $k$ for which every Boolean degree~$d$ function on $\binom{[n]}{k}$ is a junta, whenever $n \geq 2k$.
In this paper, we completely resolve this question.

\begin{theorem} \label{thm:main}
Let $d \geq 1$. There exists a constant $m(d)$ such that the following holds.

If $k \geq 2d$ then for any $n \geq 2k$, every Boolean degree~$d$ function on $\binom{[n]}{k}$ is an $m(d)$-junta.

Conversely, if $1 \leq k < 2d$ then for every $m$ there exist $n \geq 2k$ and a Boolean degree~$d$ function on $\binom{[n]}{k}$ which is not an $m$-junta.
\end{theorem}

The second part of the theorem follows from functions of the form
\[
 \sum_{i=1}^\ell \prod_{j=1}^e x_{(e-1)i + j}, \quad e = \min(d,k).
\]
When $n \geq 2\ell e$, these functions are not $\ell e$-juntas.

\paragraph{$A$-valued functions} We prove \Cref{thm:main} in the more general setting of \emph{$A$-valued} functions, for any finite $A$. These are functions $f$ such that $f(x) \in A$ for all $x \in \{0,1\}^n$. When $A = \{0,1,\ldots,a-1\}$ (or more generally, any arithmetic progression of length $a$), the junta threshold is $ad$. The situation gets more interesting when $A$ is not an arithmetic progression. For example, when $A = \{0,1,3\}$, the threshold for $d = 1$ is $k = 2$, and the threshold for $d = 2$ is $k = 6$. The latter threshold is tight due to the following example, which is $A$-valued when $k = 5$:
\[
 3 - 2\sum_{1 \leq i \leq m} x_i + \sum_{1 \leq i < j \leq m} x_i x_j.
\]

When $A$ is not an arithmetic progression, the threshold depends on a parameter first studied, in the special case of $A = \{0,1\}$, by von zur Gathen and Roche~\cite{GathenR1997}. Let $W(A,d)$ be the minimal value $W$ such that every degree~$d$ polynomial $P$ satisfying $P(0),\ldots,P(W) \in A$ is constant.

\begin{theorem} \label{thm:main-A}
Let $A$ be a finite set containing at least two elements, and let $d \geq 1$. 
There exists a constant $m(A,d)$ such that the following holds.
Define
\[
 k(A,d) = d + \max_{1 \leq s \leq d} \left( \left\lfloor \frac{d}{s} \right\rfloor (W(A,s) - s) \right),
\]
which is equal to $|A| d$ if $A$ is an arithmetic progression.

If $k \geq k(A,d)$ then for any $n \geq 2k$, every $A$-valued degree~$d$ function on $\binom{[n]}{k}$ is an $m(A,d)$-junta.

Conversely, if $1 \leq k < k(A,d)$ then for every $m$ there exist $n \geq 2k$ and an $A$-valued degree~$d$ function on $\binom{[n]}{k}$ which is not an $m$-junta.
\end{theorem}

When $A$ is an arithmetic progression, the maximum in the definition of $k(A,d)$ is obtained (not necessarily uniquely) at $s = 1$. When $A = \{0,1,3\}$ and $d = 2$, the maximum is obtained uniquely at $s = 2$.

\paragraph{The infinite slice} When $1 \leq k < 2d$, the non-junta example in the Boolean case extends to \emph{infinitely} many variables:
\[
 \sum_{i=1}^\infty \prod_{j=1}^e x_{(e-1)i + j}, \quad e = \min(d,k).
\]
The same holds for the non-junta example we gave for $A = \{0,1,3\}$ and $d = 2$. This is a general feature of our non-junta examples. We can think of such expressions as function on the \emph{infinite slice} $\binom{[\infty]}{k}$, which consists of all vectors in $\{0,1\}^{\mathbb{N}}$ of Hamming weight $k$. Conversely, when $k \geq k(A,d)$, every $A$-valued degree~$d$ function on $\binom{[\infty]}{k}$ is a junta.

\begin{theorem} \label{thm:main-A-infinite}
Let $A$ be a finite set containing at least two elements, and let $d \ge 1$. The following holds for the parameters $m(A,d),k(A,d)$ defined in \Cref{thm:main-A}.

If $k \geq k(A,d)$ then every $A$-valued degree~$d$ function on $\binom{[\infty]}{k}$ is an $m(A,d)$-junta.

Conversely, if $1 \leq k < k(A,d)$ then there exists an $A$-valued degree~$d$ function on $\binom{[\infty]}{k}$ which is not an $m$-junta for \emph{any} finite $m$.
\end{theorem}

\paragraph{Structure of the paper} After a few preliminaries in \Cref{sec:prel}, we prove our main theorems in \Cref{sec:main}. We conclude the paper with a few remarks in \Cref{sec:remarks}.

\paragraph{Acknowledgements} This project has received funding from the European Union's Horizon 2020 research and innovation programme under grant agreement No~802020-ERC-HARMONIC.

\section{Preliminaries} \label{sec:prel}

\paragraph{Slice} For integers $0 \leq k \leq n$, we define the \emph{slice} $\binom{[n]}{k}$ as
\[
 \binom{[n]}{k} = \left\{ x \in \{0,1\}^n : \sum_{i=1}^n x_i = k \right\}.
\]
We think of functions on the slice as accepting as input $n$ bits $x_1,\ldots,x_n \in \{0,1\}$, with the promise that exactly $k$ of them are equal to~$1$.

A function $f$ on the slice $\binom{[n]}{k}$ is \emph{$A$-valued}, for some $A \subseteq \mathbb{R}$, if $f(x) \in A$ for all $x \in \binom{[n]}{k}$. A \emph{Boolean} function is an $\{0,1\}$-valued function.

\paragraph{Degree}
For $S \subseteq [n] = \{1,\ldots,n\}$, we define
\[
 x_S = \prod_{i \in S} x_i,
\]
with $x_\emptyset = 1$. We call $x_S$ a \emph{degree $|S|$ monomial}.

A function on the slice $\binom{[n]}{k}$ has \emph{degree (at most) $d$} if it can be expressed as a polynomial of degree at most $d$ over the variables $x_1,\ldots,x_n$. We will usually omit the words ``at most''.

\begin{lemma} \label{lem:degree-d-repr}
If $k \geq d$, then every degree~$d$ function on $\binom{[n]}{k}$ can be expressed as a linear combination of degree~$d$ monomials.
\end{lemma}
\begin{proof}
Let $f$ be a degree~$d$ function on $\binom{[n]}{k}$. By definition, it can be expressed as a polynomial $P$ of degree at most~$d$. Since $x_i^2 = x_i$, we can replace each monomial of $P$ by its multilinearization, obtained by replacing higher powers of each $x_i$ by $x_i$, obtaining a multilinear polynomial $Q$ of degree at most~$d$ expressing $f$. Using the identity
\[
 x_S = \frac{1}{\binom{k-|S|}{d-|S|}} \sum_{\substack{S \subseteq T \subseteq [n] \\ |T| = d}} x_T,
\]
which is valid over $\binom{[n]}{k}$,
we can convert $Q$ into an equivalent polynomial in which all monomials have degree exactly~$d$.
\end{proof}

It turns out that if $n-k \geq d$ then the representation given by the lemma is unique. For this and more on the spectral perspective on functions on the slice, consult~\cite{Filmus2016a,FM2019}.

\paragraph{Junta} A function $f$ on the slice $\binom{[n]}{k}$ is a \emph{$J$-junta}, where $J \subseteq [n]$, if there is a function $g\colon \{0,1\}^J \to \mathbb{R}$ such that $f(x) = g(x|_J)$ for all $x \in \binom{[n]}{k}$; here $x|_J$ is the restriction of $x$ to the coordinates in $J$.

A function is an \emph{$m$-junta} if it is a \emph{$J$-junta} for some set $J$ of size at most $m$.

Given $x \in \binom{[n]}{k}$ and $i,j \in [n]$, we define $x^{(i\;j)}$ to be the vector obtained by switching coordinates $i$ and $j$.

\begin{lemma} \label{lem:junta-criterion}
Let $f$ be a function on the slice $\binom{[n]}{k}$. Suppose that $I,J$ are disjoint subsets of $[n]$ such that for every $i \in I$ and $j \in J$ there exists $x \in \binom{[n]}{k}$ such that $f(x) \neq f(x^{(i\;j)})$.

If $f$ is an $m$-junta then $m \geq \min(|I|,|J|)$.
\end{lemma}
\begin{proof}
Suppose that $f$ is an $m$-junta. Then there is a set $K \subseteq [n]$ of size at most $m$ and a function $g\colon \{0,1\}^K \to \mathbb{R}$ such that $f(x) = g(x|_K)$ for all $x \in \binom{[n]}{k}$. In particular, if $i,j \notin K$ then $f(x) = f(x^{(i\;j)})$ for all $x \in \binom{[n]}{k}$. This shows that either $K \supseteq I$ or $K \supseteq J$, and so $m \geq |K| \geq \min(|I|,|J|)$.
\end{proof}

The main result of~\cite{FI2019b} states that Boolean degree~$d$ functions on $\binom{[n]}{k}$ are juntas for large $k$.

\begin{theorem}[{\cite{FI2019b}}] \label{thm:FI2019b}
There exist constants $C,K>0$ such that the following holds. If $C^d \leq k \leq n - C^d$ and $f$ is a Boolean degree~$d$ function on $\binom{[n]}{k}$, then $f$ is a $KC^d$-junta.
\end{theorem}

A similar result holds for $A$-valued functions.

\begin{corollary} \label{cor:FI2019b}
For every finite set $A$ containing at least two elements there exist constants $C_A,K_A>0$ such that the following holds. If $C_A^d \leq k \leq n - C_A^d$ and $f$ is an $A$-valued degree~$d$ function on $\binom{[n]}{k}$, then $f$ is a $K_AC_A^d$-junta.
\end{corollary}
\begin{proof}
For each $a \in A$, define
\[
 f_a(x) = \prod_{\substack{b \in A \\ b \neq a}} \frac{f(x) - b}{a - b}.
\]
The function $f_a$ is a Boolean degree~$(|A|-1)d$ function, and
\[
 f(x) = \sum_{a \in A} a f_a(x).
\]
Let $C_A = C^{|A|-1}$ and $K_A = |A| K$. If $C_A^d \leq k \leq n - C_A^d$ then the theorem shows that each $f_a$ is a $KC_A^d$-junta, hence $f$ is a $K_AC_A^d$-junta.
\end{proof}

\paragraph{Infinite slice} For an integer $k \geq 0$, we define the \emph{infinite slice} $\binom{[\infty]}{k}$ as
\[
 \binom{[\infty]}{k} = \left\{ x \in \{0,1\}^{\mathbb{N}} : \sum_{i=1}^\infty x_i = k \right\}.
\]

A function $f$ on the infinite slice $\binom{[\infty]}{k}$ has \emph{degree $d$} if it can be expressed as an infinite sum of monomials of degree at most~$d$:
\[
 f(x) = \sum_{\substack{S \subseteq \mathbb{N} \\ |S| \leq d}} c(S) x_S.
\]
While the sum is infinite, all but $2^k$ of the monomials are non-zero on any given input, and therefore the sum on the right defines a real-valued function. \Cref{lem:degree-d-repr} extends to this setting.

The definition of junta and \Cref{lem:junta-criterion} extend to this setting as well.

\paragraph{Bipartite Ramsey theorem}

We assume familiarity with the classical Ramsey theorem. Our proof will also make use of a bipartite Ramsey theorem, whose simple proof we include for completeness.

\begin{theorem} \label{thm:bipartite-ramsey}
Let $c,d \in \mathbb{N}$ be parameters. For every $k \geq 1$ there exists $n \geq 1$ such that the following holds.

Suppose that $A,B$ are two disjoint sets of size $n$. Suppose furthermore that all subsets of $A \cup B$ of size $d$ are colored using one of $c$ colors. Then there exist subsets $A' \subseteq A$ and $B' \subseteq B$ of size $k$ and colors $c_0,\ldots,c_d$ such that every $T \subseteq A' \cup B'$ of size $d$ has color $c_{|T \cap A|}$.
\end{theorem}
\begin{proof}
We will prove the theorem under the assumption that $A,B$ are infinite. The finite version then follows by compactness.

Let $m$ be such that given a set $X$ of size $m$ together with a coloring of all of its subsets of size at most $d$ using $c$ colors, we can find a subset $A' \subseteq X$ of size $k$ and colors $c_0,\ldots,c_d$ such that the color of any $T \subseteq A'$ of size at most $d$ is $c_{|T|}$. Such an $m$ exists due to Ramsey's theorem.

Let $X$ be an arbitrary subset of $A$ of size $m$. Let $\chi$ be the $c$-coloring of the subsets of $A \cup B$ of size $d$. Assign every $T_B \subseteq B$ of size at most $d$ the color
\[
 T_A \mapsto \chi(T_A \cup T_B),
\]
where $T_A$ ranges over all subsets of $X$ of size $d - |T_B|$. That is, the color of $T_B$ is one of $c^{\binom{m}{d - |T_B|}}$ possible functions.
Applying Ramsey's theorem, we find an infinite subset $B' \subseteq B$ and a list of colors $c_{T_A}$, one for each $T_A \subseteq X$ of size at most $d$, such that for all $T_B \subseteq B'$ of size $d - |T_A|$, we have $\chi(T_A \cup T_B) = c_{T_A}$.

The choice of $m$ guarantees the existence of a subset $A' \subseteq X$ of size $k$ and colors $c_0,\ldots,c_d$ such that for every $T_A \subseteq A'$ of size at most $d$ and for every $T_B \subseteq B'$ of size $d - |T_A|$, we have $\chi(T_A \cup T_B) = c_{T_A} = c_{|T_A|}$.
\end{proof}

When $A,B$ are infinite, the proof above produces a subset $A' \subseteq A$ of size $k$ and an \emph{infinite} subset $B' \subseteq B$. It is natural to wonder whether we can ask for both $A'$ and $B'$ to be infinite. This is impossible in general. Indeed, let $A,B$ be two copies of $\mathbb{N}$, and color $A \times B$ using two colors as follows: $\chi(i,j) = 1$ if $i < j$ and $\chi(i,j) = 0$ otherwise. The reader can check that there are no infinite subsets $A',B'$ such that $\chi(i,j)$ is the same for all $i \in A'$ and $j \in B'$.

\section{Main theorems} \label{sec:main}

In this section we prove \Cref{thm:main,thm:main-A,thm:main-A-infinite}. Since \Cref{thm:main} is a special case of \Cref{thm:main-A}, it suffices to prove \Cref{thm:main-A,thm:main-A-infinite}. These theorems will follow from the following theorem, which is our main result.

\begin{theorem} \label{thm:main-body}
Let $A$ be a finite set containing at least two elements, and let $d \ge 1$. There exists a constant $\kappa(A,d)$, defined below, such that the following holds.

If $k \geq \kappa(A,d)$ then there exists a constant $m(A,d,k)$ such that every $A$-valued degree~$d$ function on $\binom{[n]}{k}$ is an $m(A,d,k)$-junta.

Conversely, if $1 \leq k < \kappa(A,d)$ then for every $m \geq 1$ there exist an $n$ and an $A$-valued degree~$d$ function on $\binom{[n]}{k}$ which is not an $m$-junta. Similarly, there exists an $A$-valued degree~$d$ function on $\binom{[\infty]}{k}$ which is not an $m$-junta for any finite $m$.
\end{theorem}

The constant $\kappa(A,d)$ is the smallest value $\kappa$ such that all of the following hold:
\begin{enumerate}
	\item $\kappa > d$.
	\item For all $e \in \{0,\ldots,d-1\}$: if $P$ is a univariate polynomial of degree at most $d-e$ and $P(0),\ldots,P(\kappa-e) \in A$ then $P$ is constant.
	\item For all $t \geq 0$ and $r,s \geq 1$ satisfying $t + rs \leq d$: if $P$ is a univariate polynomial of degree at most $s$ and $P(0),\ldots,P(\lfloor \frac{\kappa-t}{r} \rfloor) \in A$ then $P$ is constant.
\end{enumerate}

We show in \Cref{sec:parameter} that $\kappa(A,d)$ exists, that is, some $\kappa$ satisfies all these constraints.

Since $\kappa \geq d$, if the polynomial $P$ in Item~2 is not constant then the sequence $P(0),\ldots,P(\kappa-e)$ is not constant. For the same reason, if the polynomial $P$ in Item~3 is not constant then the sequence $P(0),\ldots,P(\lfloor \frac{k-t}{r} \rfloor)$ is not constant.

Let us explain this definition by way of proving the converse part of \Cref{thm:main-body}.

\begin{proof}[Proof of converse part of \Cref{thm:main-body}]
Let $a,b$ be two distinct elements of $A$. For each $k$ such that $1 \leq k < \kappa(A,d)$ and each $m \geq k$, we will construct $n$ and an $A$-valued degree~$d$ function on the slice $\binom{[n]}{k}$ which is not an $(m-1)$-junta. In order to prove that the function is not a junta, we will appeal to \Cref{lem:junta-criterion}.

Suppose first that $1 \leq k \leq d$. Let $n = 2km$, and consider the function
\[
 f(x) = a + (b-a) \sum_{i=1}^m x_{\{ (i-1)k + 1, \ldots, i k \}}.
\]
By construction, $f$ has degree at most $k$.
The sum is always at most $1$, and so this function is $A$-valued. Let $I = \{1,\ldots,km\}$ and $J = \{km+1,\ldots,2km\}$. For each $i' = (i-1)k + \ell \in I$ and $j \in J$, let $x \in \binom{[n]}{k}$ be given by $x_{(i-1)k+1} = \cdots = x_{ik} = 1$, and all other coordinates are zero. Then $f(x) = b$ and $f(x^{(i'\;j)}) = a$. Applying \Cref{lem:junta-criterion}, we see that $f$ is not an $(m-1)$-junta.

From now on, we assume that $k > d$.

Suppose next that $e \in \{0,\ldots,d-1\}$ and there exists a univariate polynomial $P$ of degree at most $d-e$ such that $P(0),\ldots,P(k-e) \in A$ and $P$ is non-constant. Since $k \geq d$, the list $P(0),\ldots,P(k-e)$ cannot be constant, and so $P(w) \neq P(w-1)$ for some $w \in \{1,\ldots,k-e\}$. Let $n = e + 2m$, where $m \geq k-e$, and consider the function
\[
 f(x) = a(1 - x_{\{1,\ldots,e\}}) + x_{\{1,\ldots,e\}} P\left(\sum_{i=1}^m x_{e+i}\right).
\]
By construction, $f$ has degree at most $e+(d-e) = d$. If $x_{\{1,\ldots,e\}} = 0$ then $f(x) = a$, and otherwise, the input to $P$ is at most $k - e$, and so $f$ is $A$-valued. Let $I = \{e+1,\ldots,e+m\}$ and $J = \{e+m+1,\ldots,e+2m\}$. For each $i' = i + e \in I$ and $j \in J$, let $x \in \binom{[n]}{k}$ be any input such that $x_1 = \cdots = x_e = 1$; $x_{e+h} = 1$ for exactly $w$ many $h \in \{1,\ldots,m\}$; and $x_j = 0$.
This requires $e+w \leq k$ inputs to be~$1$ and $m-w+1 \leq m$ inputs to be~$0$. Since $n-k \geq m$, such an input exists.
The input $x$ satisfies $f(x) = P(w)$ and $f(x^{(i\;j)}) = P(w-1)$. Applying \Cref{lem:junta-criterion}, we see that $f$ is not an $(m-1)$-junta.

Finally, suppose that $t \geq 0$ and $r,s \geq 1$ satisfy $t + rs \leq d$, and that there exists a univariate polynomial $P$ of degree at most $s \leq \lfloor \frac{d-t}{r} \rfloor$ such that $P(0),\ldots,P(\lfloor \frac{k-t}{r} \rfloor) \in A$ and $P$ is non-constant. Since $k \geq d$, the list $P(0),\ldots,P(\lfloor \frac{k-t}{r} \rfloor)$ cannot be constant, and so $P(w) \neq P(w-1)$ for some $w \in \{1,\ldots,\lfloor \frac{k-t}{r} \rfloor\}$. Let $n = t + 2rm$, where $m \geq k-t$, and consider the function
\[
 f(x) = a(1 - x_{\{1,\ldots,t\}}) + x_{\{1,\ldots,t\}} P\left( \sum_{i=1}^m x_{\{t+(i-1)r+1,\ldots,t+ir\}} \right).
\]
By construction, $f$ has degree at most $t + rs \leq d$. If $x_{\{1,\ldots,t\}} = 0$ then $f(x) = a$, and otherwise, the input to $P$ is at most $\frac{k-t}{r}$, and so $f$ is $A$-valued. Let $I = \{t+1,\ldots,t+rm\}$ and $J = \{t+rm+1,\ldots,t+2rm\}$. For each $i' = t + (i-1)r + \ell$ and $j \in J$, let $x \in \binom{[n]}{k}$ be given by $x_1 = \cdots = x_t = 1$; $x_{t+(h-1)r+1} = \cdots = x_{t+hr} = 1$ for exactly $w$ many $h \in \{1,\ldots,m\}$; and $x_j = 0$.
This requires $t + rw \leq k$ inputs to be~$1$ and $m-w+1 \leq m$ inputs to be~$0$. Since $n-k \geq m$, such an input exists.
The input $x$ satisfies $f(x) = P(w)$ and $f(x^{(i\;j)}) = P(w-1)$. Applying \Cref{lem:junta-criterion}, we see that $f$ is not an $(m-1)$-junta.

\smallskip

Taking $m = \infty$ and allowing for infinitely many more input coordinates, in all cases listed above we obtain $A$-valued degree~$d$ functions on $\binom{[\infty]}{k}$ which are not $m$-juntas for any finite $m$. For example, when $1 \leq k \leq d$ we can consider the function
\[
 f(x) = a + (b-a) \sum_{i=1}^\infty x_{\{2(i-1)k + 2, \ldots, 2ik\}}.
\]
For any $m$, we can take $I = \{x_2,x_4,\ldots,x_{2m}\}$ and $J = \{x_1,x_3,\ldots,x_{2m-1}\}$ and conclude, via \Cref{lem:junta-criterion}, that $f$ is not an $(m-1)$-junta.
\end{proof}

The proof of \Cref{thm:main-body} occupies \Cref{sec:quantization,sec:bunching,sec:sparsification,sec:junta}.
%
In order to complete the proof of \Cref{thm:main-A,thm:main-A-infinite}, we need the following lemma, proved in \Cref{sec:junta}.

\begin{lemma} \label{lem:parameter}
Let $A$ be a finite set containing at least two elements, and let $d \geq 1$. The parameters $\kappa(A,d)$ and $k(A,d)$, defined in \Cref{thm:main-body,thm:main-A}, are equal.

Furthermore, if $A$ is an arithmetic progression then $k(A,d) = |A|d$.
\end{lemma}

We can now prove our main theorems.

\begin{proof}[Proof of \Cref{thm:main-A}]
Given \Cref{lem:parameter}, the converse direction follows from \Cref{thm:main-body}. These two results also imply that for every $k \geq k(A,d)$ there is a constant $m(A,d,k)$ such that for any $n \geq 2k$, any $A$-valued degree~$d$ function on $\binom{[n]}{k}$ is an $m(A,d,k)$-junta.
\Cref{cor:FI2019b} shows that if $k \geq C_A^d$, $n \geq 2k$, and $f$ is an $A$-valued degree~$d$ function on $\binom{[n]}{k}$, then $f$ is a $K_A C_A^d$-junta. Therefore the theorem holds for
\[
 m(A,d) = \max(m(A,d,k(A,d)),\ldots,m(A,d,C_A^d),K_AC_A^d). \qedhere
\]
\end{proof}

\begin{proof}[Proof of \Cref{thm:main-A-infinite}]
Given \Cref{lem:parameter}, the converse direction follows from \Cref{thm:main-body}. Suppose now that $k \geq k(A,d)$ and that $f$ is an $A$-valued degree~$d$ function on $\binom{[\infty]}{k}$.

We first show that $f$ is an $m$-junta for $m = 2m(A,d)$. Suppose that this is not the case. We construct a sequence $i_1,j_1,\ldots,i_{m(A,d)+1},j_{m(A,d)+1}$ as follows. Given $i_1,j_1,\ldots,i_t,j_t$ for $t \leq m(A,d)$, since $f$ is not a $K_t$-junta for $K_t = \{i_1,j_1,\ldots,i_t,j_t\}$, we can find an input $v_{t+1} \in \binom{[\infty]}{k}$ and indices $i_{t+1},j_{t+1} \notin K_t$ such that $f(v_{t+1}) \neq f(v_{t+1}^{(i_{t+1}\;j_{t+1})})$.

Let $S_t$ be the set of $1$-indices of $v_t$, and let $f'$ be the restriction of $f$ to a finite slice obtained by zeroing out all coordinates other than the ones in
\[
 \bigcup_{t=1}^{m(A,d)+1} (S_t \cup \{i_t,j_t\}).
\]
According to \Cref{thm:main-A}, $f'$ is a $K$-junta for some $K$ of size at most $m(A,d)$. By construction, the inputs $v_1,\ldots,v_{m(A,d)+1}$ restrict to inputs on the domain of $f'$ which satisfy $f'(v_t) \neq f'(v_t^{(i_t\;j_t)})$. This means that $K$ intersects $\{i_t,j_t\}$ for all $t \in [m(A,d) + 1]$, and so $|K| > m(A,d)$. This contradiction shows that $f$ must be an $m$-junta. Therefore we can identify $f$ with an $A$-valued degree~$d$ function on $\binom{[m]}{d}$, which according to \Cref{thm:main-A} is an $m(A,d)$-junta.
\end{proof}

\subsection{Quantization} \label{sec:quantization}

Let $f$ be an $A$-valued degree~$d$ function on $\binom{[n]}{k}$, where $k \geq d$.
According to \Cref{lem:degree-d-repr}, we can represent $f$ as a linear combination of degree~$d$ monomials. In this part of the proof we show that the coefficients are \emph{quantized}, in the sense that they belong to a set $\mathfrak{C}$ depending only on $A,d,k$.

\begin{lemma} \label{lem:quantization}
For any $k \geq d \geq 1$ and finite $A \subseteq \mathbb{R}$ there exists a finite set $\mathfrak{C} \subseteq \mathbb{R}$ such that the following holds.

Let $f$ be an $A$-valued degree~$d$ function on $\binom{[n]}{k}$, where $n \geq k + d$, and suppose that
\[
 f(x) = \sum_{\substack{S \subseteq [n] \\ |S| = d}} c(S) x_S.
\]
Then all coefficients $c(S)$ belong to $\mathfrak{C}$.
\end{lemma}
\begin{proof}
Let $S \subseteq [n]$ be an arbitrary subset of size $d$, and let $I \subseteq [n]$ be an arbitrary subset of size $k$ disjoint from $S$. For every $e \in \{0,\ldots,d\}$, define
\[
 h(e) =
 \sum_{\substack{S' \subseteq S \\ |S'| = e}}
 \sum_{\substack{I' \subseteq I \\ |I'| = k-e}}
 f(S' \cup I').
\]
Each $h(e)$ is a sum of at most $2^{d+k}$ many elements from $A$, and so belongs to some finite set.

In order to express $h(e)$ in terms of coefficients $c(T)$, for $e \in \{0,\ldots,d\}$ define
\[
 \gamma(e) =
 \sum_{\substack{S' \subseteq S \\ |S'| = e}}
 \sum_{\substack{I' \subseteq I \\ |I'| = d-e}}
 c(S' \cup I').
\]
Simple combinatorics shows that
\[
 h(e) = \sum_{e'=0}^e \binom{e}{e'} \binom{k-e}{d-e} \gamma(e').
\]

Each $h(e)$ is a linear combination of $\gamma(0),\ldots,\gamma(e)$ whose coefficients depend only on $d,k$, in which the coefficient of $\gamma(e)$ is non-zero. Therefore we can express each $\gamma(e)$ as a similar linear combination of $h(0),\ldots,h(e)$. In particular, $c(S) = \gamma(d)$ is some linear combination of $h(0),\ldots,h(d)$, and so belongs to some finite set.
\end{proof}

The condition $n-k \geq d$ is necessary: if $n-k < d$ then
\[
 C\prod_{i=1}^d (1 - x_i)
\]
is a degree~$d$ polynomial which represents the zero function for any $C \in \mathbb{R}$.

As an aside, \Cref{lem:quantization} implies that the representation of \Cref{lem:degree-d-repr} is unique. Indeed, if $f = \sum_S c_1(S) x_S = \sum_S c_2(S) x_S$ are two such representations, then $f = \sum_S (\theta c_1(S) + (1-\theta) c_2(S)) x_S$ in another such representation for any real $\theta$. If $c_1(S) \neq c_2(S)$, then $\{ \theta c_1(S) + (1-\theta) c_2(S) : \theta \in \mathbb{R} \} = \mathbb{R}$, contradicting \Cref{lem:quantization} when applied to the finite set $A$ which is the range of $f$.

\subsection{Bunching of coefficients} \label{sec:bunching}

Suppose that $f$ is a degree~$d$ junta. \Cref{lem:quantization} shows that its degree~$d$ expansion is quantized. Yet it is not necessarily the case that the degree~$d$ expansion is sparse. For example, the degree~$d$ expansion of $x_{\{1,\ldots,d-1\}}$ is
\[
 \frac{1}{k-d+1} \sum_{i=d}^n x_{\{1,\ldots,d-1,i\}}.
\]
In the following steps of the proof, we gradually convert this kind of expansion into an expansion which mentions a bounded number of variables. The first step shows that the coefficients $c(S)$ in the degree~$d$ expansion are ``bunched'' in the following sense.

\begin{lemma} \label{lem:bunching}
For finite $A \subseteq \mathbb{R}$ containing at least two elements, $d \geq 1$, and $k \geq \kappa(A,d)$, there is a constant $N$ for which the following holds.

Let $f$ be an $A$-valued degree~$d$ function on $\binom{[n]}{k}$, where $n \geq k + d$, and suppose that
\[
 f(x) = \sum_{\substack{S \subseteq [n] \\ |S| = d}} c(S) x_S
\]
is the expansion whose existence is guaranteed by \Cref{lem:degree-d-repr}.

We can assign each subset $T \subseteq [n]$ of size smaller than $d$ a value $c(T) \in \mathfrak{C}$ (where $\mathfrak{C}$ is the set promised by \Cref{lem:quantization}) such that $c(T \cup \{i\}) = c(T)$ for all but $N$ many $i \in [n] \setminus T$.
\end{lemma}

The proof of \Cref{lem:bunching} proceeds by backwards induction on the size of the set $T$. The bulk of the work lies in the basis of the induction.

\begin{proof}[Proof of \Cref{lem:bunching}, base case]
Under the assumptions of \Cref{lem:bunching}, we assign for each subset $T \subseteq [n]$ of size $d-1$ a value $c(T) \in \mathfrak{C}$ such that $c(T \cup \{i\}) = c(T)$ for all but $N_{d-1}$ many $i \in [n] \setminus T$, where $N_{d-1}$ is a constant depending only on $A,d,k$.

Fix a subset $T \subseteq [n]$ of size $d-1$. We partition $[n] \setminus T$ into $|\mathfrak{C}|$ sets $X_\gamma$ as follows: $X_\gamma$ contains all $i \notin T$ such that $c(T \cup \{i\}) = \gamma$. For every $\gamma_1 \neq \gamma_2$, we color all non-empty subsets $S \subseteq X_{\gamma_1} \cup X_{\gamma_2}$ of size at most $d$ as follows: the color assigned to $S$ is
\[
 T' \mapsto c(T' \cup S),
\]
where $T'$ ranges over all subsets of $T$ of size $d - |S|$. According to \Cref{lem:quantization}, the color of $S$ is one of $|\mathfrak{C}|^{\binom{d-1}{d-|S|}}$ possible functions. Applying \Cref{thm:bipartite-ramsey} repeatedly, there is a constant $M$, depending only on $A,d,k$, such that if $|X_{\gamma_1}|,|X_{\gamma_2}| \geq M$ then there exist subsets $X'_{\gamma_1} \subseteq X_{\gamma_1}$ and $X'_{\gamma_2} \subseteq X_{\gamma_2}$ of size $k$ and colors $c_{T',e} \in \mathfrak{C}$, for all $T' \subseteq T$ and $e \leq d - |T'|$, such that if $S \subseteq T \cup X'_{\gamma_1} \cup X'_{\gamma_2}$ has size $d$ then $c(S) = c_{S \cap T, S \cap X_{\gamma_1}}$.

We now prove that for every $T' \subseteq T$ there exists a color $c_{T'} \in \mathfrak{C}$ such that $c_{T',e} = c_{T'}$ for all $e \leq d - |T'|$. The proof is by induction on $|T'|$. Suppose that the claim holds for all proper subsets of some $T' \subseteq T$. We prove it for $T'$.

Let $w \leq k-|T'|$. The value of $f$ on an input consisting of $T'$ together with $w$ elements from $X'_{\gamma_1}$ and $k-|T'|-w$ elements from $X'_{\gamma_2}$ is
\[
 \sum_{T'' \subsetneq T'} \binom{k - |T''|}{d - |T''|} c_{T''} +
 \sum_{e = 0}^{d - |T'|} \binom{w}{e} \binom{k-|T'|-w}{d-|T'|-e} c_{T',e}.
\]
This is a polynomial $P(w)$ of degree at most $d - |T'|$ such that $P(0),\ldots,P(k - |T'|) \in A$, and so since $k \geq \kappa(A,d)$, $P$ is constant.

Since $P(e)$ only depends on $c_{T',0},\ldots,c_{T',e}$, it follows that for every $w \in \{1,\ldots,d-|T'|\}$ we have
\[
 P(w) - P(w-1) = \binom{k-|T'|-w}{d-|T'|-w} c_{T',w} - \sum_{e=0}^{w-1} \rho_{w,e} c_{T',e},
\]
for some $\rho_{w,0},\ldots,\rho_{w,w-1}$. If $c_{T',0} = c_{T',1} = \cdots = c_{T',w} = c_{T'}$ then $P(w) = P(w-1)$ since both are equal to $\sum_{T'' \subseteq T'} \binom{k-|T''|}{d-|T''|} c_{T''}$. This shows that $\sum_e \rho_{w,e} = \binom{k-|T'|-w}{d-|T'|-w}$.

We can now prove inductively that $c_{T',w} = c_{T',0}$ for $w \in \{1,\ldots,d-|T'|\}$. Suppose that this holds for $w' < w$. Then $0 = P(w) - P(w-1) = \binom{k-|T'|-w}{d-|T'|-w} (c_{T',w} - c_{T',0})$, and so $c_{T',w} = c_{T',0}$. We can therefore take $c_{T'} = c_{T',0}$.



\smallskip

Any $i_1 \in X'_{\gamma_1}$ satisfies $\gamma_1 = c(T \cup \{i_1\}) = c_{T,1}$. Similarly, any $i_2 \in X'_{\gamma_2}$ satisfies $\gamma_2 = c(T \cup \{i_2\}) = c_{T,0}$. Since $\gamma_1 \neq \gamma_2$ whereas $c_{T,0} = c_{T,1}$, we reach a contradiction. 
It follows that at most one of the sets $X_\gamma$ can satisfy $|X_\gamma| \geq M$. Choosing $c(T')$ to be the value $\gamma$ which maximizes $|X_\gamma|$, the base case follows, with $N_{d-1} = |\mathfrak{C}|M$.
\end{proof}

The inductive step is more elementary.

\begin{proof}[Proof of \Cref{lem:bunching}, inductive step]
Let $e \leq d-2$. Suppose that each subset $T \subseteq [n]$ of size $e+1$ is assigned a value $c(T) \in \mathfrak{C}$ such that $c(T \cup \{i\}) = c(T)$ for all but $N_{e+1}$ many $i \in [n] \setminus T$. We assign for each subset $T \subseteq [n]$ of size $e$ a value $c(T) \in \mathfrak{C}$ such that $c(T \cup \{i\}) = c(T)$ for all but $N_e$ many $i \in [n] \setminus T$, where $N_e = |\mathfrak{C}|(N_{e+1}^2 + N_{e+1} + 1)$.

Fix a subset $T \subseteq [n]$ of size $e$. For $\gamma \in \mathfrak{C}$, let $X_\gamma$ consist of all $i \in [n] \setminus T$ such that $c(T \cup \{i\}) = \gamma$. In order to prove the inductive step, it suffices to show that at most one $\gamma \in \mathfrak{C}$ satisfies $|X_\gamma| \geq N_{e+1}^2 + N_{e+1} + 1$.

Suppose, for the sake of contradiction, that $|X_{\gamma_1}|, |X_{\gamma_2}| \geq N_{e+1}^2 + N_{e+1} + 1$ for some $\gamma_1 \neq \gamma_2$. Choose $N_{e+1} + 1$ arbitrary elements $i_1,\ldots,i_{N_{e+1}+1} \in X_{\gamma_1}$. By assumption, for each $i_s$ there is an exceptional set $E_s$ of size at most $N_{e+1}$ such that if $j \in [n] \setminus (T \cup \{i_s\} \cup E_s)$ then $c(T \cup \{i_s,j\}) = c(T \cup \{i_s\}) = \gamma_1$. Since $|X_{\gamma_2}| > (N_{e+1} + 1) N_{e+1}$, there exists $j \in X_{\gamma_2}$ which does not belong to any $E_s$, and consequently $c(T \cup \{j,i_s\}) = \gamma_1$ for all $s \in \{1,\ldots,N_{e+1} + 1\}$. However, this contradicts the promise that $c(T \cup \{j,i\}) = c(T \cup \{j\}) = \gamma_2$ for all but $N_{e+1}$ many $i \in [n] \setminus (T \cup \{j\})$.
\end{proof}

\Cref{lem:bunching} follows by taking $N = \max(N_0,\ldots,N_{d-1})$.

\subsection{Sparsification} \label{sec:sparsification}

If $c(S) \neq 0$ for some $S$ of size $d-1$, then we can sparsify the expansion of $f$ by introducing the appropriate product of $x_S$. In this way, we can recover $x_{\{1,\ldots,d-1\}}$ from its degree~$d$ expansion.
The following lemma carries out this procedure for all sets of size smaller than $d$.

\begin{lemma} \label{lem:sparsification}
For finite $A \subseteq \mathbb{R}$ containing at least two elements, $d \geq 1$, and $k \geq \kappa(A,d)$, there is a constant $M$ and a finite subset $\mathfrak{D}$ for which the following holds.

Let $f$ be an $A$-valued degree~$d$ function on $\binom{[n]}{k}$, where $n \geq k + d$. Then $f$ has an expression of the form
\[
 f(x) = \sum_{\substack{S \subseteq [n] \\ |S| \leq d}} C(S) x_S,
\]
where $C(S) \in \mathfrak{D}$, and for every $T \subseteq [n]$ of size less than $d$, we have $C(T \cup \{i\}) = 0$ for all but at most $M$ many $i \in [n] \setminus T$.
\end{lemma}
\begin{proof}
The transformation proceeds in several stages, and accordingly, for each $e \leq d$ we will construct a constant $M_e$, a finite subset $\mathfrak{D}_e$ (both depending only on $A,d,k$), and coefficients $c_e(S) \in \mathfrak{D}_e$ for all sets $S \subseteq [n]$ of size at most $d$, such that
\[
 f(x) = \sum_{\substack{S \subseteq [n] \\ |S| < e \text{ or } |S| = d}} \binom{k - |S|}{d - |S|} c_e(S) x_S
\]
and the following properties hold:
\begin{enumerate}[(a)]
\item For every $T \subseteq [n]$ of size less than $e$, we have $c_e(T \cup \{i\}) = 0$ for all but at most $M_e$ many $i \in [n] \setminus T$.
\item For every $T \subseteq [n]$ of size between $e$ and $d-1$, we have $c_e(T \cup \{i\}) = c_e(T)$ for all but at most $M_e$ many $i \in [n] \setminus T$.
\end{enumerate}

Once we prove that, taking $M = M_d$, $\mathfrak{D} = \mathfrak{D}_d$ and $C(S) = \binom{k-|S|}{d-|S|} c_d(S)$ will prove the lemma.

\smallskip

When $e = 0$, \Cref{lem:bunching} shows that we can take $M_0 = N$, $\mathfrak{D}_0 = \mathfrak{C}$, and $c_0 = c$.

Now suppose that we have constructed $M_e,\mathfrak{D}_e,c_e$, where $e < d$. We define $c_{e+1}(S) = c_e(S)$ if $|S| \leq e$, and
\[
 c_{e+1}(S) = c_e(S) - \sum_{\substack{T \subseteq S \\ |T| = e}} c_e(T)
\]
if $|S| > e$. Since the sum on the right contains at most $2^d$ terms, we can construct the finite subset $\mathfrak{D}_{e+1}$ from the finite subset $\mathfrak{D}_e$. Next, let us check that the new coefficients represent $f$:
\begin{multline*}
\sum_{\substack{S \subseteq [n] \\ |S| \leq e \text{ or } |S| = d}} \binom{k - |S|}{d - |S|} c_{e+1}(S) x_S	= \\
\sum_{\substack{S \subseteq [n] \\ |S| < e}} \binom{k - |S|}{d - |S|} c_e(S) x_S +
\sum_{\substack{T \subseteq [n] \\ |T| = e}} \binom{k - e}{d - e} c_e(T) x_T +
\sum_{\substack{S \subseteq [n] \\ |S| = d}} \left(c_e(S) - \sum_{\substack{T \subseteq S \\ |T| = e}} c_e(T) \right) x_S = \\
\sum_{\substack{S \subseteq [n] \\ |S| < e}} \binom{k - |S|}{d - |S|} c_e(S) x_S +
\sum_{\substack{T \subseteq [n] \\ |T| = e}} c_e(T) \sum_{\substack{T \subseteq S \subseteq [n] \\ |S| = d}} x_S +
\sum_{\substack{S \subseteq [n] \\ |S| = d}} \left(c_e(S) - \sum_{\substack{T \subseteq S \\ |T| = e}} c_e(T) \right) x_S = \\
\sum_{\substack{S \subseteq [n] \\ |S| < e}} \binom{k - |S|}{d - |S|} c_e(S) x_S +
\sum_{\substack{S \subseteq [n] \\ |S| = d}} c_e(S) x_S = f(x).
\end{multline*}

It remains to prove properties~(a) and~(b). Property~(a) follows for sets of size less than $e$ by induction. If $T \subseteq [n]$ has size $e$ and $c_{e+1}(T \cup \{i\}) \neq 0$ for some $i \in [n] \setminus T$ then since
\[
 c_{e+1}(T \cup \{i\}) = 
 c_e(T \cup \{i\}) - c_e(T) - \sum_{\substack{R \subseteq T \\ |R| = e-1}} c_e(R \cup \{i\}),
\]
either $c_e(T \cup \{i\}) \neq c_e(T)$ or $c_e(R \cup \{i\}) \neq 0$ for some subset $R \subseteq T$ of size $e-1$. Property~(b) of $c_e$ shows that there are at most $M_e$ many $i \notin T$ such that $c_e(T \cup \{i\}) \neq c_e(T)$. For each $R$, property~(a) of $c_e$ shows that there are at most $M_e$ many $i \notin T$ such that $c_e(R \cup \{i\}) \neq 0$. In total, we deduce that $c_{e+1}(T \cup \{i\}) = 0$ for all but at most $(e+1)M_e$ indices $i \notin T$.

The proof of property~(b) is similar. If $T \subseteq [n]$ has size at least $e+1$ and $c_{e+1}(T \cup \{i\}) \neq c_{e+1}(T)$ then since
\[
 c_{e+1}(T \cup \{i\}) - c_{e+1}(T) =
 c_e(T \cup \{i\}) - c_e(T) + \sum_{\substack{R \subseteq T \\ |R| = e-1}} c_e(R \cup \{i\}),
\]
either $c_e(T \cup \{i\}) \neq c_e(T)$ or $c_e(R \cup \{i\}) \neq 0$ for some $R \subseteq T$ of size $e-1$ not including $i$. Property~(b) of $c_e$ shows that there are at most $M_e$ many $i \notin T$ such that $c_e(T \cup \{i\}) \neq c_e(T)$. For each $R$, property~(a) of $c_e$ shows that there are at most $M_e$ many $i \notin T$ such that $c_e(R \cup \{i\}) \neq 0$. In total, we deduce that $c_{e+1}(T \cup \{i\}) = c_{e+1}(T)$ for all but at most $2^dM_e$ indices $i \notin T$.

We complete the proof of the inductive step by taking $M_{e+1} = 2^d M_e$.
\end{proof}

\subsection{Junta conclusion} \label{sec:junta}

\Cref{lem:sparsification} gives us an expression for $f$ in which the coefficients $C(S)$ are locally sparse: for each $T$, only a bounded number of coefficients $C(T \cup \{i\})$ are non-zero. We would like to extend this to global sparsity: only a bounded number of coefficients $C(S)$ are non-zero. We do so in steps, proving the following lemma inductively.

\begin{lemma} \label{lem:junta}
For any finite $A \subseteq \mathbb{R}$ containing at least two elements, $d \geq 1$, and $k \geq \kappa(A,d)$, and any $t + r \leq d$, there exist constants $N(t,r) \geq k + d$ and $L(t,r)$ such that the following holds.

Let $f$ be an $A$-valued degree~$d$ function on $\binom{[n]}{k}$, where $n \geq N(t,r)$. Let $C(S)$ be the coefficients of the expression in \Cref{lem:sparsification}. For any subset $T \subseteq [n]$ of size $t$, there at most $L(t,r)$ many subsets $R \subseteq [n] \setminus T$ of size $r$ such that $C(T \cup R) \neq 0$.
\end{lemma}

Before proving the lemma, let us briefly show how it implies the main part of \Cref{thm:main-body} (we proved the converse part at the beginning of \Cref{sec:main}).

\begin{proof}[Proof of main part of \Cref{thm:main-body}]
We prove the theorem with
\[
 m(A,d,k) = \max\left(N(0,1),\ldots,N(0,d), \sum_{r=1}^d r L(0,r)\right).
\]
Let $f$ be an $A$-valued degree~$d$ function on $\binom{[n]}{k}$, where $k \geq \kappa(A,d)$. If $n < N(0,r)$ for some $r \in \{1,\ldots,d\}$, then $f$ is trivially an $n$-junta, and so an $m(A,d,k)$-junta. Otherwise, consider the expression promised by \Cref{lem:sparsification}:
\[
 f(x) = \sum_{\substack{S \subseteq [n] \\ |S| \leq d}} C(S) x_S.
\]
According to \Cref{lem:junta}, for all $r \in \{1,\ldots,d\}$, at most $L(0,r)$ many sets $S \subseteq [n]$ of size $r$ satisfy $C(S) \neq 0$. If we take the union of all these sets for all $r$, we obtain a set $J$ of size at most $m(A,d,k)$ such that $f$ is a $J$-junta, completing the proof.
\end{proof}

We now turn to the proof of \Cref{lem:junta}.

\begin{proof}[Proof of \Cref{lem:junta}]
When $r = 0$, the lemma trivially holds, for $N(t,0) = k + d$ and $L(t,0) = 1$. When $r = 1$, the lemma follows directly from \Cref{lem:sparsification}, taking $N(t,1) = k + d$ and $L(t,1) = M$. Therefore we can assume that $r \geq 2$.

We prove the lemma for all other parameters by induction: first on $r$, then on $t$. This means that given $t,r$, we assume that the lemma holds for all $(t',r')$ such that $r' < r$ and for all $(t',r)$ such that $t' < t$, and prove it for $(t,r)$.

Let us be given $t,r$ such that $t + r \leq d$ and $r \geq 2$, and let $T \subseteq [n]$ be a set of size $t$. We want to bound the size of the collection $\mathcal{R}$ consisting of all subsets of $[n]$ of size $r$ which are disjoint from $T$ and satisfy $C(T \cup R) \neq 0$. We will show that for the correct choice of $N(t,r) \geq t+r$ and $L(t,r)$, the assumption $|\mathcal{R}| \geq L(t,r)$ leads to a contradiction. It follows that $|\mathcal{R}| < L(t,r)$.

Starting with $\mathcal{R}$, we will extract subcollections $\mathcal{R} \supseteq \mathcal{R}_1 \supseteq \mathcal{R}_2 \supseteq \mathcal{R}_3 \supseteq \mathcal{R}_4$ which are more and more structured:
\begin{itemize}
\item All $R \in \mathcal{R}_1$ are \emph{good}: $C(S) = 0$ for all subsets $S \subseteq T \cup R$ intersecting $R$ other than $T \cup R$ itself.
\item The sets in $\mathcal{R}_2$ are disjoint.
\item If $R_1,\ldots,R_s \in \mathcal{R}_3$ are such that $C(S) \neq 0$ for some subset $S \subseteq T \cup R_1 \cup \cdots \cup R_s$ intersecting $R_1,\ldots,R_s$ and different from $T \cup R_i$ then $|S \cap T| + rs \leq d$.
\item For all $T' \subseteq T$ and all $R_1,\ldots,R_s \in \mathcal{R}_4$, the sum of $C(T' \cup S)$ over all subsets $S \subseteq R_1 \cup \cdots \cup R_s$ intersecting $R_1,\ldots,R_s$ only depends on $T'$ and $s$.
\end{itemize}

Choosing $L(t,r)$ large enough, we will be able to guarantee that $|\mathcal{R}_4| \geq k$. Choosing $N(t,r)$ large enough, we will be able to find $k$ many points $P$ outside of $T,\mathcal{R}_4$ such that $C(S) = 0$ for any $S \subseteq T \cup \bigcup \mathcal{R}_4 \cup P$ intersecting $P$, and this will enable us to reach a contradiction.

\smallskip

We now proceed with the details. Rephrasing the above definition, a set $R \in \mathcal{R}$ is good if $C(T' \cup R') = 0$ for all $T' \subseteq T$ and non-empty $R' \subseteq R$, other than $T' = T$ and $R' = R$. In order to show that many sets are good, we bound the number of sets which are bad.

Let $T' \subsetneq T$ be a set of size $t' < t$. According to the induction hypothesis, the number of $R' \subseteq [n]$ of size $r' \in \{1,\ldots,r\}$ disjoint from $T'$ such that $C(T' \cup R') \neq 0$ is at most $L(t',r')$. Applying the induction hypothesis again, for each such $R'$, the number of sets $R'' \subseteq [n]$ of size $r - r'$ disjoint from $T' \cup R'$ such that $C(T' \cup R' \cup R'') \neq 0$ is at most $L(t'+r', r-r')$. Every set $R \in \mathcal{R}$ which is bad due to $T' \neq T$ is of the form $R' \cup R''$, and so for each $T'$, there are at most $L(t',r') L(t'+r',r-r')$ such sets.

If $T' = T$ then the same argument works as long as $r' < r$. It follows that the number of bad sets is at most
\[
 \Lambda' = \sum_{t'=0}^{t-1} \binom{t}{t'} \sum_{r'=1}^r L(t',r') L(t'+r',r-r') + \sum_{r'=1}^{r-1} L(t,r') L(t+r',r-r').
\]
Accordingly, if we define $\mathcal{R}_1$ to consist of all good $R \in \mathcal{R}$, then $|\mathcal{R}_1| \geq \Lambda_1 := L(t,r) - \Lambda'$.

\smallskip

The next step is constructing $\mathcal{R}_2$. To that end, consider a graph whose vertices are the sets in $\mathcal{R}_1$, and in which two vertices $R_1,R_2$ are connected if they are not disjoint. We will show that the graph has bounded degree, and so a large independent set.

If $R_1,R_2 \in \mathcal{R}_1$ are not disjoint then there is some $i \in R_1$ such that $i \in R_2$. Given $i \in R_1$, the induction hypothesis shows that the number of possible $R_2$ is $L(t+1,r-1)$, since $R_2 \setminus \{i\}$ is a subset of $[n]$ of size $r-1$, disjoint from $T \cup \{i\}$, such that $C(T \cup \{i\} \cup (R_2 \setminus \{i\})) \neq 0$. Since there are $r$ choices for $i$, this shows that the degree of every vertex in the graph is at most $r L(t+1,r-1)$.

A simple greedy algorithm now constructs a subset $\mathcal{R}_2 \subseteq \mathcal{R}_1$ of size at least $\Lambda_2 := \Lambda_1/(r L(t+1,r-1)+1)$.

\smallskip

In order to construct $\mathcal{R}_3$, we consider a hypergraph on the vertex set $\mathcal{R}_2$. For each $T' \subseteq T$ and $s \leq d$ such that $|T'| + rs > d$, we add a hyperedge $\{R_1,\ldots,R_s\}$ (where all $R_i$ are different) if there exist non-empty $R'_i \subseteq R_i$ such that $C(T' \cup R'_1 \cup \cdots \cup R'_s) \neq 0$ (we define $C(S) = 0$ if $|S| > d$). We will show that this graph contains few hyperedges, specifically at most $K_{t,r}|\mathcal{R}_2|^{s-1}$ hyperedges of uniformity $s$. 

Let $T' \subseteq T$ have size $t'$ and let $s \leq d$ be such that $t' + rs > d$. We want to bound the number of sets $\{R_1,\ldots,R_s\}$ (where all $R_i$ are different) such that $C(T' \cup R'_1 \cup \cdots \cup R'_s) \neq 0$ for some non-empty $R'_i \subseteq R_i$. If $R'_i = R_i$ for all $i$ then $|T' \cup R'_1 \cup \cdots \cup R'_s| = t' + rs > d$, and so $C(T' \cup R'_1 \cup \cdots \cup R'_s) = 0$. Therefore $R'_i \neq R_i$ for some $i$. By rearranging the indices, we can assume that $R'_s \neq R_s$.

There are at most $|\mathcal{R}_2|^{s-1}$ many choices for $R_1,\ldots,R_{s-1}$. For each choice of distinct $R_1,\ldots,R_{s-1}$, there are at most $2^{sr}$ many choices of non-empty $R'_1,\ldots,R'_{s-1}$. Given $R'_1,\ldots,R'_{s-1}$ of combined size $u$ and given $r' \in \{1,\ldots,r-1\}$, the induction hypothesis shows that there are at most $L(t'+u,r')$ many sets $R'_s \subseteq [n]$ of size $r'$, disjoint from $T' \cup R'_1 \cup \cdots \cup R'_{s-1}$, such that $C(T' \cup R'_1 \cup \cdots \cup R'_s) \neq 0$. For each such $R'_s$, the induction hypothesis shows that there are at most $L(t+r',r-r')$ many sets $R''_s \subseteq [n]$ disjoint from $T \cup R'_s$ such that $C(T \cup R'_s \cup R''_s) \neq 0$. Altogether, the number of hyperedges of uniformity $s$ is at most
\[
 \sum_{t'=0}^t \binom{t}{t'} |\mathcal{R}_2|^{s-1} 2^{sr} \sum_{u=0}^d \sum_{r'=1}^{r-1} L(t'+u,r') L(t+r',r-r'),
\]
where $L(t',r') = 0$ if $t' + r' > d$.
Hence we can find a constant $K_{t,r}$ (depending on known $L(t',r')$) such that the number of hyperedges of uniformity $s$ is at most $K_{t,r}|\mathcal{R}_2|^{s-1}$.

Suppose now that we sample a subset of $\mathcal{R}_2$ by including each $R \in \mathcal{R}_2$ with probability $p = |\mathcal{R}_2|^{-(1-1/d)}$, and then removing all $R$ which are incident to any surviving hyperedge. The expected number of surviving $R$ is at least
\[
 p |\mathcal{R}_2| - \sum_{s=1}^d s p^s K_{t,r} |\mathcal{R}_2|^{s-1} =
 |\mathcal{R}_2|^{1/d} - K_{t,r} \sum_{s=1}^d s |\mathcal{R}_2|^{s/d-1} \geq
 |\mathcal{R}_2|^{1/d} - K_{t,r} d^2.
\]
In particular, we can find a subset $\mathcal{R}_3$ of size at least $\Lambda_3 := \Lambda_2^{1/d} - K_{t,r}d^2$ which spans no hyperedges. That is, if $R_1,\ldots,R_s \in \mathcal{R}_3$ and $C(S) \neq 0$ for some $S \subseteq T \cup R_1 \cup \cdots \cup R_s$ intersecting all of $R_1,\ldots,R_s$, then $|S \cap T| + rs > d$.

\smallskip

We construct $\mathcal{R}_4$ by applying Ramsey's theorem. For every $s$ such that $rs \leq d$, we color every subset $\{R_1,\ldots,R_s\} \subseteq \mathcal{R}_3$ of size $s$ by the function
\[
 T' \mapsto \sum_{\substack{R'_1 \subseteq R_1,\ldots,R'_s \subseteq R_s \\ R'_1,\ldots,R'_s \neq \emptyset}} C(T' \cup R'_1 \cup \cdots \cup R'_s).
\]
where $T'$ ranges over all subsets of $T$ (recall that we defined $C(S) = 0$ when $|S| > d$). According to \Cref{lem:quantization}, all summands belong to a finite set $\mathcal{D}$, and so the sum attains one of at most $|\mathcal{D}|^{2^{rs}}$ possible values. Consequently, the number of colors is at most $(2^t)^{|\mathcal{D}|^{2^{rs}}}$. If $\mathcal{R}_3$ is large enough then we can apply Ramsey's theorem to obtain a subset $\mathcal{R}_4 \subseteq \mathcal{R}_3$ of size $k$, and values $\Gamma(T',s)$ for all $T' \subseteq T$ and $s \leq \lfloor d/r \rfloor$, such that all distinct $R_1,\ldots,R_s \in \mathcal{R}_4$ satisfy
\[
 \sum_{\substack{R'_1 \subseteq R_1,\ldots,R'_s \subseteq R_s \\ R'_1,\ldots,R'_s \neq \emptyset}} C(T' \cup R'_1 \cup \cdots \cup R'_s) = \Gamma(T',s).
\]
We can extend the definition of $\Gamma$ to larger $s$. The construction of $\mathcal{R}_3$ guarantees that $\Gamma(T',s) = 0$ if $|T'| + rs > d$. Moreover, since all $R \in \mathcal{R}_4$ are good, we know that $\Gamma(T',1) = 0$ if $T' \neq T$ and $\Gamma(T,1) \neq 0$.

\smallskip

At this point, we can explain how to choose $L(t,r)$. We choose $L(t,r)$ so that the condition $|\mathcal{R}_3| \geq \Lambda_3$ is strong enough in order for the application of Ramsey's theorem detailed above to go through. 

\smallskip

Let $V$ consist of the union of all sets in $\mathcal{R}_4$.
The next step is to choose a set $P = \{ p_1,\ldots, p_k \} \subseteq [n]$ of size $k$ such that $C(S) = 0$ for any subset $S \subseteq T \cup V \cup P$ intersecting $P$. This will be possible assuming that $n$ is large enough.

We choose $P$ in $k$ steps. In the $i$'th step, given the choice of $p_1,\ldots,p_{i-1}$, we choose $p_i$. For any $e < d$ and any subset $S' \subseteq T \cup V \cup \{p_1,\ldots,p_{i-1}\}$ of size $e$, there are at most $L(e,1)$ many $p \notin S'$ such that $C(S' \cup \{p\}) \neq 0$. Therefore we can find a suitable $p_i$ as long as
\[
 n > N_i(t,r) := t + kr + i-1 + \sum_{e=0}^{d-1} \binom{t + kr + i-1}{e} L(e,1).
\]
Accordingly, we choose $N(t,r) = \max(k+d,N_{k-1}(t,r)+1)$. This ensures that we can choose the set $P$.

\smallskip

Let $T'$ be an inclusion-minimal subset of $T$ such that $\Gamma(T',s) \neq 0$ for some $s > 0$, and let $t' = |T'|$. This means that $\Gamma(T'',s) = 0$ for all $T'' \subsetneq T'$ and $s > 0$. Such a choice is possible since $\Gamma(T,1) \neq 0$. Also, let $s' > 0$ be the minimal value such that $\Gamma(T',s') \neq 0$.

Let $w$ be such that $t' + rw \leq k$. The value of $f$ on an input consisting of $T'$ together with the union of $w$ sets from $\mathcal{R}_4$ and $k - t' - rw$ elements from $P$ is
\[
 \sum_{T'' \subseteq T'} \sum_{s=0}^d \binom{w}{s} \Gamma(T'',s) =
 \sum_{T'' \subseteq T'} \Gamma(T'',0) + \sum_{s=s'}^{\lfloor \frac{d-t'}{r} \rfloor} \binom{w}{s} \Gamma(T',s).
\]
This is a polynomial $Q(w)$ of degree at most $\lfloor \frac{d-t'}{r} \rfloor$ such that $Q(0),\ldots,Q(\lfloor \frac{k-t'}{r} \rfloor) \in A$, and so since $k \geq \kappa(A,d)$, $Q$ is constant. However, by construction, $Q(s') - Q(s'-1) = \Gamma(T',s') \neq 0$. We have reached the required contradiction, completing the proof.
\end{proof}

\subsection{The parameter \texorpdfstring{$k(A,d)$}{k(A,d)}} \label{sec:parameter}

In this subsection we show that $k(A,d) = \kappa(A,d)$, and prove that $k(A,d) = |A|d$ when $A$ is an arithmetic progression, thus proving \Cref{lem:parameter}. We start by giving an alternative formula for $\kappa(A,d)$ in terms of the parameter $W(A,d)$ introduced in \Cref{sec:introduction}, which is the minimal value $W$ such that every degree $d$ polynomial $P$ satisfying $P(0),\ldots,P(W) \in A$ is constant.

Before giving the formula for $\kappa(A,d)$ in terms of $W(A,d)$, let us show that $W(A,d)$ is indeed well-defined.

\begin{lemma} \label{lem:W-finite}
If $A \subseteq \mathbb{R}$ is a set containing at least two elements and $d \geq 1$ then $d < W(A,d) \leq |A|d$.
\end{lemma}
\begin{proof}
Suppose that $P$ is a degree~$d$ polynomial. We will show that if $P(0),\ldots,P(W) \in A$ for $W = |A|d$ then $P$ is constant, and so $W(A,d) \leq |A| d$. According to the pigeonhole principle, there is $a \in A$ such that $P(i) = a$ for at least $d+1$ many $i \in \{0,\ldots,W\}$. Since every non-constant degree~$d$ polynomial has at most $d$ roots, we conclude that $P$ is constant.

In order to show that $W(A,d) > d$, we will exhibit a non-constant degree~$d$ polynomial $P$ satisfying $P(0),\ldots,P(d) \in A$. Let $a,b \in A$ be two distinct elements of $A$. We define
\[
 P(x) = a + (b - a) \prod_{i=0}^{d-1} \frac{x - i}{d - i}.
\]
By construction, $P(0) = \cdots = P(d-1) = a$ and $P(d) = b$.
\end{proof}

Here is the formula for $\kappa(A,d)$ in terms of $W(A,d)$. It is the minimal $\kappa$ which satisfies the following conditions:
\begin{enumerate}
\item $\kappa \ge d + 1$.
\item $\kappa-e \ge W(A,d-e)$ for all $e \in \{0,\ldots,d-1\}$.
\item $\lfloor \frac{\kappa-t}{r} \rfloor \geq W(A,s)$ whenever $r,s \geq 1$ and $t + rs \leq d$.
\end{enumerate}

This results in the following formula, whose proof is immediate.

\begin{lemma} \label{lem:kappa-def}
If $A \subseteq \mathbb{R}$ is a finite set containing at least two elements and $d \ge 1$ then
\[
 \kappa(A,d) = \max\left(d + 1, \max_{0 \leq e \leq d-1} e + W(A,d-e), \max_{\substack{1 \leq s \leq d \\ 1 \leq r \leq \lfloor d/s \rfloor}} d - rs + r W(A,s) \right).
\]
\end{lemma}

Using this formula, we can prove \Cref{lem:parameter}.

\begin{proof}[Proof of \Cref{lem:parameter}]
\Cref{lem:W-finite} shows that $W(A,d) \geq d+1$. Consequently, $0 + W(A,d-0) \geq d + 1$, and so we can drop the first term in the formula in \Cref{lem:kappa-def}. Taking $s = d-e$ and $r = 1$, the third term recovers the second term. Therefore
\[
 \kappa(A,d) =
 \max_{\substack{1 \leq s \leq d \\ 1 \leq r \leq \lfloor d/s \rfloor}} d + r (W(A,s) - s) = \max_{1 \leq s \leq d} d + \left\lfloor \frac{d}{s} \right\rfloor (W(A,s) - s),
\]
since $W(A,s) \geq s+1$ according to \Cref{lem:W-finite}. The expression on the right-hand side coincides with the formula for $k(A,d)$ in the statement of \Cref{thm:main-A}.

Suppose now that $A$ is an arithmetic progression, say $A = \{a, a+b, \ldots, a+(m-1)b\}$, where $m = |A|$. The polynomial $P(x) = a + bx$ shows that $W(A,1) > |A|-1$, and so $W(A,1) = |A|$ according to \Cref{lem:W-finite}. Taking $s = 1$ in the formula for $k(A,d)$, this shows that $k(A,d) \geq d + d(|A| - 1) = |A|d$. On the other hand, for every $s \in \{1,\ldots,d\}$ we have
\[
 d + \left\lfloor \frac{d}{s} \right\rfloor (W(A,s) - s) \leq
 d + \frac{d}{s} (s|A| - s) = |A|d,
\]
using \Cref{lem:W-finite}. Therefore $k(A,d) = |A|d$.
\end{proof}

When $A$ is not an arithmetic progression, it is not necessarily the case that $k(A,d) = |A|d$. For example, $k(A,1) = W(A,1)$ is the length of the longest arithmetic progression contained in~$A$.

Here are the values of $W(A,d),k(A,d)$ for several choices of $A$:
\[
\begin{array}{c|*{5}{c}|*{10}{c}}
\multirow{2}{*}{A} & \multicolumn{5}{c|}{W(A,d)} & \multicolumn{10}{c}{k(A,d)} \\\cline{2-16}
 & 1 & 2 & 3 & 4 & 5 & \multicolumn{2}{c}{1} & \multicolumn{2}{c}{2} & \multicolumn{2}{c}{3} & \multicolumn{2}{c}{4} & \multicolumn{2}{c}{5}\\\hline
\{0,1\} & 2 & 4 & 4 & 6 & 6 & 2 & [1] & 4 & [1,2] & 6 & [1] & 8 & [1,2] & 10 & [1] \\
\{0,1,3\} & 2 & 6 & 6 & 7 & 8 & 2 & [1] & 6 & [2] & 7 & [2] & 12 & [2] & 13 & [2] \\
\{0,1,4,5,20\} & 2 & 5 & 7 & 8 & 8 & 2 & [1] & 5 & [2] & 7 & [3] & 10 & [2] & 11 & [2] \\
\{0,1,27,126,370\} & 2 & 4 & 4 & 10 & 10 & 2 & [1] & 4 & [1,2] & 6 & [1] & 10 & [4] & 11 & [4]
\end{array}
\]
The numbers in squares indicate that values of $s$ for which $k(A,d)$ is attained.



\section{Final remarks} \label{sec:remarks}

\paragraph{Another threshold} \Cref{thm:FI2019b}, proved in~\cite{FI2019b}, states that if $C^d \leq k \leq n-C^d$ and $f$ is a Boolean degree~$d$ function on $\binom{[n]}{k}$, then $f$ is a $KC^d$-junta. The result proved in~\cite{FI2019b} is in fact stronger: under the same assumptions, there is a Boolean degree~$d$ function $g$ \emph{on the Boolean cube} $\{0,1\}^n$ such that $f$ is the restriction of $g$ to the slice. This implies the junta conclusion since every Boolean degree~$d$ function on the Boolean cube is an $O(2^d)$-junta~\cite{NisanSzegedy,CHS,Wellens}.

In this paper, we answer one open question raised in~\cite{FI2019b}: we find the minimal $k = k(d)$ such that every Boolean degree~$d$ function on $\binom{[n]}{k}$, where $n \geq 2k$, is a junta. Another open question in~\cite{FI2019b} asks for the minimal $\ell = \ell(d)$ such that every Boolean degree~$d$ function on $\binom{[n]}{\ell}$, where $n \geq 2\ell$, is the restriction of a Boolean degree~$d$ function on $\{0,1\}^n$. 
Clearly, $\ell(d) \geq k(d)$. Is it the case that $\ell(d) = k(d)$? When $d = 1$, this follows from~\cite{FI2019a}.

More generally, we can define $\ell(A,d)$ for any finite $A$. It is not always the case that $\ell(A,d) = k(A,d)$. For example, if $A = \{0,5,7,8,12,13,15\}$ then $k(A,1) = 2$ whereas $\ell(A,1) = 3$. Indeed, the function $5x_1 + 7x_2 + 8x_3$ is $A$-valued on $\binom{[n]}{2}$ for any $n \geq 4$, but is not the restriction of any $A$-valued degree~$1$ function on $\{0,1\}^n$.

\paragraph{Multislice} The \emph{multislice} is the generalization of the slice to functions on $\{0,\ldots,m-1\}$ for arbitrary $m$. Given a partition $n = \lambda_0 + \cdots + \lambda_{m-1}$, the corresponding multislice consists of all vectors in $\{0,\ldots,m-1\}^n$ containing exactly $\lambda_i$ coordinates whose value is $i$. Given another partition $k = k_1 + \cdots + k_{m-1}$, we can consider the family of multislices with $\lambda_0 \geq k$ and $\lambda_1 = k_1,\ldots,\lambda_{m-1} = k_{m-1}$. We conjecture that all of our results extend to this setting.

\bibliographystyle{alpha}
\bibliography{biblio}

\end{document}